\documentclass[preprint,12pt,authoryear]{elsarticle}  

\usepackage{amsmath} 
\usepackage{bm}
\usepackage{amsthm}
\usepackage{hypernat}
\usepackage[hidelinks]{hyperref}
\usepackage[T1]{fontenc}
\usepackage{booktabs}
\usepackage{extarrows}
\usepackage{enumerate}
\usepackage{enumitem} 
\usepackage{mathtools}
\usepackage{xcolor}
\usepackage{xfrac}
\usepackage{graphicx}
\usepackage{amssymb}
\usepackage{natbib}
\setcitestyle{numbers,square}
\usepackage{mathrsfs}

\everymath{\displaystyle}

\theoremstyle{definition}

\theoremstyle{plain}

\newtheorem{thm}{Theorem}[section]
\newtheorem*{cthm}{Theorem}
 \newtheorem{cor}[thm]{Corollary}
 \newtheorem{lem}[thm]{Lemma}
 \newtheorem{prop}[thm]{Proposition}
 \theoremstyle{definition}
 
 \theoremstyle{remark}

 \numberwithin{equation}{section}

\newenvironment{re}{\textbf{Remark.}}{\hfill $\qedsymbol$ \par}
\hypersetup{colorlinks=true, citecolor=cyan, linkcolor=cyan, urlcolor=cyan}
\newcommand{\C}{\mathbb{C}}
\newcommand{\R}{\mathbb{R}}
\newcommand{\N}{\mathbb{N}}

\journal{}

\begin{document}
	\begin{frontmatter}
		
		
		\title{Banach Hardy--Sobolev Spaces on the Upper Half-plane and Operator Theory}
		
		
		\author[author1]{Haoxian Liang}
		\author[author2]{Haichou Li}
		\author[author3]{Tao Qian}
		
		\address[author1]{College of Mathematics and Informatics, South China Agricultural University, Guangzhou, China. Email: 934585789@qq.com}
		\address[author2]{Corresponding author. College of Mathematics and Informatics, South China Agricultural University, Guangzhou, China. Email: hcl2016@scau.edu.cn}
		\address[author3]{Faculty of Innovation Engineering, Macau University of Science and Technology, Macau, China. Email: tqian@must.edu.mo} 


\begin{abstract}
We study Hardy--Sobolev spaces $H_n^p(\C^+)$ on the upper half-plane for
$1\le p\le\infty$ and $n\in\mathbb N$, from both function-theoretic and
operator-theoretic viewpoints. We establish an isometric boundary
characterization of $H_n^p(\C^+)$ via nontangential limits, together with a
Sobolev-type embedding theorem, a Cauchy integral representation, a
direct-sum decomposition of $W_n^p(\R)$ for $1<p<\infty$, and a generalized
Banach algebra structure under pointwise multiplication. We also obtain a
finer Fourier-analytic description in the Hilbert case $p=2$ by proving a
Paley--Wiener theorem and deriving the reproducing kernel of $H_n^2(\C^+)$.
On the operator-theoretic side, we prove the spectral formula for
multiplication operators and establish two verifiable sufficient conditions
for the boundedness of weighted composition operators. These results provide
a systematic theory of Hardy--Sobolev spaces on the upper half-plane beyond
the Hilbert setting.
\end{abstract}

		
        
	\end{frontmatter}

	\tableofcontents

	
	\section{Introduction}

Hardy--Sobolev spaces are natural analytic function spaces that combine the
boundary control of Hardy spaces with the additional regularity encoded by
Sobolev norms. They occupy an important position in complex analysis,
harmonic analysis, and operator theory. On bounded domains such as the unit
ball $\mathbb{B}^n$ $(n\ge 1)$, the theory of Hardy--Sobolev spaces is by now
well developed. For the spaces $H_s^p(\mathbb{B}^n)$, Carleson measures were
characterized in~\cite{a10,a11}, and a variety of operator-theoretic problems
have been studied, including Toeplitz and Hankel operators, Fredholmness,
essential spectra, multipliers, composition operators, and related integral
operators; see, for instance,
\cite{a12,a13,a14,a15,a7,a18,a19,a16,a17} and the references therein.

By contrast, the Hardy--Sobolev theory on unbounded domains remains much less
developed, even in one complex variable. A notable feature of the available
literature is that many results are concentrated in the Hilbert case $p=2$,
where reproducing kernel Hilbert space techniques and Fourier--Laplace methods
provide powerful structural tools. Outside the Hilbert setting, however,
reproducing kernels are no longer available as a general mechanism, and it is
far from clear whether the existing theory can be extended to the full Banach
range $1\le p\le\infty$. This leads to the central problem of the present
paper: how to develop a systematic Hardy--Sobolev theory on the upper
half-plane beyond the Hilbert framework, and what can replace reproducing
kernels in that setting.

In this paper we study Hardy--Sobolev spaces on the upper half-plane
\[
\C^+=\{z\in\C: Im(z)>0\}.
\]
Let $H(\C^+)$ denote the space of holomorphic functions on $\C^+$, and let
$H^p(\C^+)$ be the classical Hardy space on $\C^+$ for $1\le p\le\infty$.
For $n\in\N$ and $1\le p\le\infty$, we define
\[
H_n^p(\C^+)=\Bigl\{F\in H(\C^+):F^{(k)}\in H^p(\C^+),\ k=0,1,\dots,n\Bigr\},
\]
equipped with the norm
\[
\|F\|_{H_n^p}=
\begin{cases}
\Bigl(\sum_{k=0}^n\|F^{(k)}\|_{H^p}^p\Bigr)^{1/p}, & 1\le p<\infty,\\[6pt]
\sum_{k=0}^n\|F^{(k)}\|_{H^\infty}, & p=\infty.
\end{cases}
\]
These spaces provide a natural half-plane analogue of classical
Hardy--Sobolev spaces on bounded domains, while also reflecting the special
interaction among holomorphic boundary values on $\mathbb R$, derivative
regularity, and Fourier analysis in the upper half-plane.

Several earlier works already indicate the richness of the half-plane
Hardy--Sobolev setting. Early appearances of related Sobolev--Hardy
structures may be found in the study of rational wavelets in~\cite{a21}.
Dang--Qian--You~\cite{a2} and Dang--Qian--Yang~\cite{a22} studied
Hardy--Sobolev derivatives and obtained, in particular, the direct-sum
decomposition
\[
W_1^2(\mathbb R)=H_1^2(\C^+)\oplus H_1^2(\C^-)
\]
in the sense of isometric isomorphism. In a more general Hilbert-space
Laplace-transform framework, Kucik~\cite{a1} introduced the spaces
$A_{(m)}^2$ and proved spectral inclusion results for multiplication
operators. Another Hilbert Hardy--Sobolev space on the right half-plane,
\[
\mathscr{H}_n^2(\C_+)=\Bigl\{F\in H(\C_+):z^kF^{(k)}\in H^2(\C_+),\
k=0,1,\dots,n\Bigr\},
\]
was investigated in~\cite{5,4}, where Paley--Wiener type theorems,
reproducing kernels, and boundedness criteria for weighted composition
operators were established. These works show that half-plane Hardy--Sobolev
spaces form a natural and fruitful framework, but they also indicate that
much of the existing theory is still essentially Hilbertian.

The main purpose of this paper is to develop a systematic theory of
$H_n^p(\C^+)$ in the full Banach range $1\le p\le\infty$, from both
function-theoretic and operator-theoretic viewpoints. Our basic principle is
that, beyond the Hilbert case, the structural role played by reproducing
kernels should be replaced by boundary theory. More precisely, we show that
nontangential boundary values provide an isometric identification of
$H_n^p(\C^+)$ with a natural Hardy--Sobolev space on $\mathbb R$, denoted by
$H_n^p(\mathbb R)$. This boundary model is the starting point of the paper
and serves as the main mechanism from which the subsequent function-theoretic
and operator-theoretic results are derived.

Our first main contribution is the Banach-range structure theory of
$H_n^p(\C^+)$. In Section~3 we prove an isometric boundary characterization
of $H_n^p(\C^+)$ via nontangential limits, together with a Sobolev-type
embedding estimate and a Cauchy integral representation. For
$1<p<\infty$, we further obtain a direct-sum decomposition of
$W_n^p(\mathbb R)$ into upper and lower half-plane Hardy--Sobolev parts,
thereby extending the known Hilbert-space decomposition to the full
reflexive Banach range. We also show that, for $n\ge1$, the space
$H_n^p(\C^+)$ admits a generalized Banach algebra structure under pointwise
multiplication. This phenomenon has no counterpart in the classical Hardy
space case $n=0$ and plays a crucial bridging role in the operator-theoretic
part of the paper.

Our second main contribution concerns the special Hilbert case $p=2$. In
Section~4 we prove a Paley--Wiener theorem for $H_n^2(\C^+)$ and derive an
explicit reproducing kernel. These results provide a finer Fourier-analytic
description of the Hilbertian Hardy--Sobolev space and complement the
boundary-based Banach theory developed earlier.

Our third main contribution lies in operator theory. In Section~5 we study
multiplication operators and weighted composition operators on
$H_n^p(\C^+)$. For multipliers $\psi\in\mathscr{M}_{n,p}$, we prove the sharp
spectral formula
\[
\sigma(T_\psi)=\overline{\psi(\C^+)},
\]
which extends the known spectral inclusion in the Hilbert setting to an exact
spectral description in our Banach framework. As an immediate consequence,
the only compact multiplication operator on $H_n^p(\C^+)$ is the zero
operator. We also establish two verifiable sufficient conditions for the
boundedness of weighted composition operators on $H_n^p(\C^+)$.

We emphasize that the novelty of the paper lies not only in the statements of
these results, but also in the method. In the non-Hilbert setting, our
arguments do not rely on reproducing kernel techniques; instead, they are
based on boundary behavior, weak derivative structure, Sobolev embedding, and
Cauchy-type representations. This provides a unified framework for treating
Hardy--Sobolev spaces on the upper half-plane throughout the full Banach
range.

The principal results of the paper are Theorems~\ref{t1}, \ref{t3}, and
\ref{t4} in Section~3, Theorems~\ref{t2} and \ref{t5} in Section~4, and
Theorems~\ref{t7}, \ref{t8}, and \ref{t9} in Section~5.

The paper is organized as follows. Section~2 collects preliminaries on
Hardy spaces, Sobolev spaces, and weighted Lebesgue spaces. Section~3
develops the Banach-range theory of $H_n^p(\C^+)$, including the boundary
characterization, Sobolev-type embedding, Cauchy integral representation,
direct-sum decomposition, and generalized Banach algebra structure.
Section~4 is devoted to the Hilbert case $p=2$, where we establish the
Paley--Wiener theorem and derive the reproducing kernel of $H_n^2(\C^+)$.
Finally, Section~5 treats multiplication operators and weighted composition
operators on $H_n^p(\C^+)$.


    \section{Preliminaries}

In this section we collect the notation and standard facts that will be used
throughout the paper. We recall basic properties of Hardy spaces on the
upper half-plane, Sobolev spaces on the real line, and certain weighted
Lebesgue spaces on $\R^+$ that arise naturally in the Fourier-analytic
description of $H_n^2(\C^+)$.

\subsection{Hardy Spaces on the Half-Plane}

For $0<p\le\infty$, the Hardy space $H^p(\C^+)$ consists of all holomorphic
functions $F$ on the upper half-plane such that
\[
	\|F\|_{H^p}
	=
	\begin{cases}
		\displaystyle
		\sup_{y>0}\left(\int_{-\infty}^{\infty}|F(x+iy)|^p\,dx\right)^{1/p}<\infty,
		& 0<p<\infty,\\[12pt]
		\displaystyle
		\sup_{z\in\C^+}|F(z)|<\infty,
		& p=\infty.
	\end{cases}
\]

For $1\le p<\infty$, every function $F\in H^p(\C^+)$ admits nontangential
boundary values $F_l$ almost everywhere on $\R$. The boundary function
belongs to the Hardy boundary space
\[
	\begin{split}
		H^p(\R)
		:&=
		\{\,f\in L^p(\R): f \text{ is the nontangential boundary value of }
		F\in H^p(\C^+)\,\} \\
		&=
		\left\{\,f\in L^p(\R):
		\int_{-\infty}^{\infty}\frac{f(x)}{x-\overline z}\,dx=0,\ \forall z\in\C^+
		\right\}.
	\end{split}
\]
If $p=\infty$, we write
\[
	H^\infty(\R)
	=
	\left\{\,f\in L^\infty(\R):
	\int_{-\infty}^{\infty}\frac{f(x)}{(x-\overline z)(x+i)}\,dx=0,\
	\forall z\in\C^+
	\right\},
\]
and $F$ is uniquely determined by $F_l$. Moreover,
\[
	\|F\|_{H^p}=\|F_l\|_{L^p}.
\]

For $1\le p<\infty$, each $F\in H^p(\C^+)$ admits the Cauchy integral
representation
\begin{equation}
	\label{e8}
	F(z)=\frac{1}{2\pi i}\int_{-\infty}^{\infty}\frac{F_l(x)}{x-z}\,dx,
	\qquad z\in\C^+.
\end{equation}
Consequently,
\[
	F(x+iy)
	=
	\frac{1}{\pi}\int_{-\infty}^{\infty}
	F_l(t)\,\frac{y}{(x-t)^2+y^2}\,dt.
\]

For $f\in L^p(\R)$ with $1<p<\infty$, define
\[
	F_+(z)=\frac{1}{2\pi i}\int_{-\infty}^{\infty}\frac{f(x)}{x-z}\,dx,
	\qquad z\in\C^+,
\]
and
\[
	F_-(z)=-\frac{1}{2\pi i}\int_{-\infty}^{\infty}\frac{f(x)}{x-z}\,dx,
	\qquad z\in\C^-.
\]
Their nontangential boundary values are
\[
	\frac12 f+\frac{i}{2}Hf
	\qquad\text{and}\qquad
	\frac12 f-\frac{i}{2}Hf,
\]
respectively, where $H$ denotes the Hilbert transform
\[
	Hf(x)=\frac{1}{\pi}\lim_{\varepsilon\to0}
	\int_{|x-t|>\varepsilon}\frac{f(t)}{x-t}\,dt.
\]
These formulas are consequences of the Plemelj theorem; see, for instance,
\cite[Corollary~14.8]{6} or \cite[Theorem~1.2.1]{19}. When $p=1$, the same
boundary relations remain valid, although $F_\pm$ need not belong to
$H^1(\C^\pm)$.

We write $\prescript{+}{}{H^p}(\R)$ and $\prescript{-}{}{H^p}(\R)$ for the
boundary Hardy spaces corresponding to the upper and lower half-planes,
respectively. The preceding boundary formulas imply the classical
decomposition
\begin{equation}
	\label{e4}
	L^p(\R)=\prescript{+}{}{H^p}(\R)\oplus\prescript{-}{}{H^p}(\R),
	\qquad 1<p<\infty,
\end{equation}
where the sum is direct because
\[
	\prescript{+}{}{H^p}(\R)\cap \prescript{-}{}{H^p}(\R)=\{0\}.
\]
Moreover, $\prescript{+}{}{H^2}(\R)$ and $\prescript{-}{}{H^2}(\R)$ are
orthogonal in $L^2(\R)$.

In the Hilbert case $p=2$, the classical Paley--Wiener theorem states that
the holomorphic Fourier transform
\[
	\mathscr{F}(f)(z)=\int_0^\infty f(x)e^{izx}\,dx,
	\qquad z\in\C^+,
\]
defines an isometric isomorphism from $L^2(\R^+,2\pi\,dx)$ onto
$H^2(\C^+)$; see \cite[Theorem~19.2]{3}. We refer to \cite{6,17,19} for
further background on Hardy spaces.

\subsection{Sobolev Spaces on the Real Line}

Let $L^1_{\mathrm{loc}}(\R)$ denote the space of locally integrable
functions on $\R$, and let $C_c^\infty(\R)$ denote the space of test
functions. A function $F\in L^1_{\mathrm{loc}}(\R)$ is said to be weakly
differentiable if there exists a function $G\in L^1_{\mathrm{loc}}(\R)$ such
that
\[
	\int_{-\infty}^{\infty} F(x)\varphi'(x)\,dx
	=
	-\int_{-\infty}^{\infty} G(x)\varphi(x)\,dx
\]
for all $\varphi\in C_c^\infty(\R)$. In this case we write $D(F)=G$ and
call $G$ the weak derivative of $F$.

Define inductively
\[
	W_1^0(\R)
	=
	\{\,F:\R\to\C : F \text{ is weakly differentiable}\,\},
\]
and, for $n\ge2$,
\[
	W_n^0(\R)
	=
	\{\,F\in W_{n-1}^0(\R): D^{(n-1)}F\in W_1^0(\R)\,\}.
\]

The following standard characterization will be used repeatedly; see, for
example, \cite[Theorem~8.2]{11}.

\begin{prop}
	\label{p4}
	Let $n\in\mathbb N^+$. A function $F$ belongs to $W_n^0(\R)$ if and only if
	there exists a finite sequence $\{F_k\}_{k=0}^n\subset L^1_{\mathrm{loc}}(\R)$
	such that $F=F_0$ almost everywhere and
	\[
		F_k(b)-F_k(a)=\int_a^b F_{k+1}(x)\,dx
	\]
	for all $-\infty<a<b<\infty$ and $k=0,\dots,n-1$.
\end{prop}

For $1\le p\le\infty$ and $n\in\mathbb N^+$, the Sobolev space
$W_n^p(\R)$ is defined by
\[
	W_n^p(\R)
	=
	\{\,F\in L^p(\R)\cap W_n^0(\R): D^{(k)}F\in L^p(\R),\ k=1,\dots,n\,\},
\]
equipped with the Sobolev norm
\[
	\|F\|_{W_n^p}
	=
	\begin{cases}
		\displaystyle
		\left(\sum_{k=0}^n \|D^{(k)}F\|_{L^p}^p\right)^{1/p},
		& 1\le p<\infty,\\[12pt]
		\displaystyle
		\sum_{k=0}^n \|D^{(k)}F\|_{L^\infty},
		& p=\infty.
	\end{cases}
\]

Proposition~\ref{p4} immediately yields the following criterion.

\begin{cor}
	\label{c1}
	Let $n\in\mathbb N^+$ and $1\le p\le\infty$. A function $F:\R\to\C$
	belongs to $W_n^p(\R)$ if and only if there exists a finite sequence
	$\{F_k\}_{k=0}^n\subset L^p(\R)$ such that $F=F_0$ almost everywhere and
	\[
		F_k(b)-F_k(a)=\int_a^b F_{k+1}(x)\,dx
	\]
	for all $-\infty<a<b<\infty$ and $k=0,\dots,n-1$.
\end{cor}

We refer to \cite{11,15,18} for background on Sobolev spaces and to
\cite{a1,a3} for related weighted Fourier-side constructions.

\subsection{Weighted Lebesgue Spaces on $\R^+$}

For $1\le p<\infty$ and $n\in\mathbb N$, define the measure
\[
	d\mu_{n,p}=2\pi x^{np}\,dx
\]
on $\R^+=(0,\infty)$. We write $L^p(\R^+,d\mu_{n,p})$ for the
corresponding weighted Lebesgue space, with norm
\[
	\|F\|_{\mu_{n,p}}
	=
	\left(\int_0^\infty |F(x)|^p\,2\pi x^{np}\,dx\right)^{1/p}.
\]

Set
\[
	L_n^p(\R^+)
	=
	\bigcap_{k=0}^n L^p(\R^+,d\mu_{k,p}),
\]
equipped with the norm
\[
	\|F\|_{L_n^p}
	=
	\left(\sum_{k=0}^n \|F\|_{\mu_{k,p}}^p\right)^{1/p}.
\]
Equivalently, $L_n^p(\R^+)=L^p(\R^+,dv_{n,p})$ with equivalent norm, where
\[
	dv_{n,p}=2\pi(1+x^p)^n\,dx.
\]

The following elementary equivalence will be used in Section~4.

\begin{prop}
	\label{p3}
	For $1\le p<\infty$ and $n\in\mathbb N$,
	\[
		L_n^p(\R^+)=L^p(\R^+,d\mu_{0,p})\cap L^p(\R^+,d\mu_{n,p}).
	\]
\end{prop}

By Proposition~\ref{p3}, the norm
\[
	\|F\|_{n,p}
	=
	\left(\|F\|_{\mu_{0,p}}^p+\|F\|_{\mu_{n,p}}^p\right)^{1/p}
\]
is equivalent to $\|\cdot\|_{L_n^p}$.

The next embedding will also be used in the Hilbert case.

\begin{prop}
	\label{c3}
	If $n\ge1$, then
	\[
		L_n^p(\R^+)\hookrightarrow L_1^p(\R^+)\hookrightarrow L^1(\R^+).
	\]
\end{prop}

When $p=2$, the space $L_n^2(\R^+)$ is a Hilbert space under the inner
product
\[
	\langle f,g\rangle_{L_n^2}
	=
	\sum_{k=0}^n \int_0^\infty f(x)\overline{g(x)}\,2\pi x^{2k}\,dx.
\]
This Hilbert-space structure will be used in Section~4.

\section{Function Theory of $H_n^p(\C^+)$}

This section contains the principal new function-theoretic results of the
paper in the full Banach range $1\le p\le\infty$. Our starting point is an
isometric boundary characterization of $H_n^p(\C^+)$ via nontangential
limits. Its significance is not merely that it extends the classical Hardy
boundary theory to the Sobolev setting, but that it provides a
boundary-based structural model for $H_n^p(\C^+)$ beyond the Hilbert
framework. From this model we derive several fundamental consequences,
including Sobolev-type embedding estimates, a Cauchy integral
representation, a direct-sum decomposition of $W_n^p(\R)$ for
$1<p<\infty$, and, for $n\ge1$, a generalized Banach algebra structure
under pointwise multiplication. The latter will also serve as a key bridge
to the operator-theoretic part of the paper.

\subsection{Boundary Characterization of $H_n^p(\C^+)$}

For the classical Hardy space $H^p(\C^+)$, nontangential boundary values
provide an isometric identification with the Hardy boundary space
$\prescript{+}{}{H^p}(\R):=H^p(\R)$. The first main result of this section shows that
an analogous statement remains valid at the Hardy--Sobolev level. The point
is that, for $n\ge1$, one must not only recover the boundary values of the
function itself, but also identify the Sobolev structure carried by its
derivatives. This leads naturally to the boundary Hardy--Sobolev space
\[
H_n^p(\R):=\prescript{+}{}{H^p}(\R)\cap W_n^p(\R),
\]
equipped with the Sobolev norm.

\begin{thm}
	\label{t1}
	Let $1\le p\le\infty$ and $n\in\mathbb N$. Then the nontangential boundary
	value map defines an isometric isomorphism from $H_n^p(\C^+)$ onto
	$H_n^p(\R)$.
\end{thm}

The proof combines the classical Hardy boundary theory with the weak
derivative characterization of Sobolev spaces and repeated integration by
parts. In this way, one can recover the full boundary Sobolev structure of
$H_n^p(\C^+)$ without relying on reproducing-kernel or Hilbert-space
arguments.

\begin{proof}
	The case $n=0$ is exactly the classical boundary theory of Hardy spaces on
	the upper half-plane. Assume now that $n\ge1$ and let
	$F\in H_n^p(\C^+) \subset H^p(\C^+)$. Denote by $F_l$ the nontangential
	boundary value of $F$. Since $F\in H^p(\C^+)$, we have
	\[
		\int_{-\infty}^{\infty}\frac{F_l(x)}{x-\overline z}\,dx=0,
		\qquad z\in\C^+.
	\]
	Moreover, for every $k=0,1,\dots,n-1$,
	\begin{equation}
		\label{e1-new}
		F^{(k)}(b+iy)-F^{(k)}(a+iy)
		=
		\int_a^b F^{(k+1)}(x+iy)\,dx
	\end{equation}
	for all $-\infty<a<b<\infty$.

	Let
	\[
		F_l^{(k)}(x)=\lim_{y\to0^+}F^{(k)}(x+iy),
	\]
	which exists for almost every $x\in\R$ because $F^{(k)}\in H^p(\C^+)$ for
	$k=0,1,\dots,n$. For fixed $a<b$, H\"older's inequality gives
	\[
		\left|
		\int_a^b F^{(k)}(x+iy)\,dx-\int_a^b F_l^{(k)}(x)\,dx
		\right|
		\le
		(b-a)^{1/q}
		\|F^{(k)}(\cdot+iy)-F_l^{(k)}\|_{L^p(\R)},
	\]
	with the usual interpretation when $p=1$. Since
	$F^{(k)}(\cdot+iy)\to F_l^{(k)}$ in $L^p(\R)$ as $y\to0^+$, it follows
	that
	\[
		\int_a^b F^{(k)}(x+iy)\,dx \to \int_a^b F_l^{(k)}(x)\,dx.
	\]
	Passing to the limit in \eqref{e1-new}, we obtain
	\begin{equation}
		\label{e7}
		F_l^{(k)}(b)-F_l^{(k)}(a)=\int_a^b F_l^{(k+1)}(x)\,dx
	\end{equation}
	for almost every $a<b$ and every $k=0,1,\dots,n-1$. By
	Corollary~\ref{c1}, this shows that $F_l\in W_n^p(\R)$. Together with the
	Hardy boundary condition above, we conclude that $F_l\in H_n^p(\R)$.
	Moreover,
	\begin{equation}
		\label{e2-new}
		F_l^{(k)}=D^{(k)}F_l \qquad \text{a.e. on }\R,\quad k=0,1,\dots,n.
	\end{equation}

	Conversely, let $G\in H_n^p(\R)$. For $1\le p<\infty$, define
	\[
		F(z)=\frac{1}{2\pi i}\int_{-\infty}^{\infty}\frac{G(x)}{x-z}\,dx,
		\qquad z\in\C^+.
	\]
	Then $F\in H^p(\C^+)$ and its nontangential boundary value equals $G$
	almost everywhere. For $k=1,\dots,n$, differentiating under the integral
	sign yields
	\[
		F^{(k)}(z)=\frac{k!}{2\pi i}\int_{-\infty}^{\infty}
		\frac{G(x)}{(x-z)^{k+1}}\,dx.
	\]
	Using integration by parts, equivalently the weak derivative structure of
	$G$, we obtain
	\[
		F^{(k)}(z)=\frac{(-1)^{k-1}}{2\pi i}\int_{-\infty}^{\infty}
		\frac{D^{(k-1)}G(x)}{(x-z)^2}\,dx.
	\]
	Since $G\in H_n^p(\R)$, we also have
	\[
		0
		=
		k!\int_{-\infty}^{\infty}\frac{G(x)}{(x-\overline z)^{k+1}}\,dx
		=
		(-1)^{k-1}\int_{-\infty}^{\infty}
		\frac{D^{(k-1)}G(x)}{(x-\overline z)^2}\,dx.
	\]
	Subtracting the two expressions gives
	\[
		F^{(k)}(z)
		=
		\frac{(-1)^k}{2\pi i}
		\int_{-\infty}^{\infty} D^{(k-1)}G(t)
		\left[
		\frac{1}{(t-\overline z)^2}-\frac{1}{(t-z)^2}
		\right]dt.
	\]
	Writing $z=x+iy$, we note that
	\[
		\frac{\partial}{\partial t}\frac{2iy}{(x-t)^2+y^2}
		=
		\frac{1}{(t-\overline z)^2}-\frac{1}{(t-z)^2}.
	\]
	Hence, integrating by parts once more,
	\[
		F^{(k)}(x+iy)
		=
		\frac{(-1)^{k+1}}{\pi}
		\int_{-\infty}^{\infty}
		D^{(k)}G(t)\,\frac{y}{(x-t)^2+y^2}\,dt.
	\]
	The right-hand side is the Poisson integral of $(-1)^{k+1}D^{(k)}G$, so
	$F^{(k)}\in H^p(\C^+)$ for all $k=0,1,\dots,n$, and therefore
	$F\in H_n^p(\C^+)$. Moreover, by the classical boundary isometry for Hardy
	spaces and \eqref{e2-new},
	\[
		\|F^{(k)}\|_{H^p}
		=
		\|F_l^{(k)}\|_{L^p}
		=
		\|D^{(k)}G\|_{L^p},
		\qquad k=0,1,\dots,n.
	\]
	Thus the boundary-value map is an isometry.

	For $p=\infty$, one uses the normalized Cauchy integral
	\[
		F(z)=\frac{1}{2\pi i}\int_{-\infty}^{\infty}
		G(x)\left(\frac{1}{x-z}-\frac{1}{x+i}\right)\,dx,
		\qquad z\in\C^+,
	\]
	and argues similarly. The subtraction of $(x+i)^{-1}$ ensures absolute
	convergence, and the same boundary and derivative arguments show that
	$F\in H_n^\infty(\C^+)$ with boundary value $G$.
\end{proof}

As immediate consequences of Theorem~\ref{t1}, we record several basic
structural properties of $H_n^p(\C^+)$.

\begin{cor}
	\label{c2a}
	Let $1\le p\le\infty$ and $0\le m\le n$. Then
	\[
		H_n^p(\C^+) \hookrightarrow H_m^p(\C^+).
	\]
\end{cor}

\begin{cor}
	\label{c2b}
	For every $1\le p\le\infty$ and $n\in\mathbb N$, the space
	$H_n^p(\C^+)$ is a Banach space.
\end{cor}
\begin{proof}
	For $1\le p<\infty$, by Theorem~\ref{t1} it suffices to prove that
	$H_n^p(\R)$ is complete. Let $\{F_j\}\subset H_n^p(\R)$ be Cauchy. Since
	$W_n^p(\R)$ is complete, there exists $G\in W_n^p(\R)$ such that
	\[
		\lim_{j\to\infty}\|F_j-G\|_{W_n^p}=0.
	\]
	It remains to show that $G\in H_n^p(\R)$. Fix $z\in\C^+$. By H\"older's
	inequality,
	\[
		\left|\int_{-\infty}^{\infty}\frac{G(x)}{x-\overline{z}}\,dx\right|
		=
		\left|\int_{-\infty}^{\infty}\frac{G(x)-F_j(x)}{x-\overline{z}}\,dx\right|
		\le \|G-F_j\|_{L^p}
		\left\|\frac{1}{\,\cdot-\overline{z}\,}\right\|_{L^q(\R)},
	\]
	with the usual interpretation when $p=1$. Letting $j\to\infty$ gives
	\[
		\int_{-\infty}^{\infty}\frac{G(x)}{x-\overline{z}}\,dx=0,
	\]
	so $G\in H_n^p(\R)$.

	For $p=\infty$, let $\{F_j\}$ be Cauchy in $H_n^\infty(\C^+)$. Then for
	each $k=0,1,\dots,n$, the sequence $\{F_j^{(k)}\}$ converges uniformly on
	$\C^+$ to some $F_k\in H^\infty(\C^+)$. Standard arguments show that
	$F_k=F_0^{(k)}$ for all $k$, so $F_0\in H_n^\infty(\C^+)$. Hence
	$H_n^\infty(\C^+)$ is complete.
\end{proof}

\begin{cor}
	\label{c2c}
	Let $F\in H_n^p(\C^+)$ with $n\ge1$. Then, for each integer
	$0\le k\le n-1$, the derivative $F^{(k)}$ extends continuously to
	$\overline{\C^+}$.
\end{cor}
\begin{proof}
	By \eqref{e7} we have
	\[
		F_l^{(k)}(x)=F_l^{(k)}(a)+\int_a^x F_l^{(k+1)}(t)\,dt,
	\]
	so $F_l^{(k)}$ is continuous for every $0\le k\le n-1$. Therefore
	$F^{(k)}$ is the Poisson integral of a continuous boundary function and
	extends continuously to $\overline{\C^+}$.
\end{proof}

Analogously, one may define Hardy--Sobolev spaces on the lower half-plane,
denoted by $H_n^p(\C^-)$. The corresponding boundary value space is denoted
by $\prescript{-}{}{H_n^p}(\R):=\prescript{-}{}{H^p}(\R)\cap W_n^p(\R)$.
All of the above statements have natural counterparts in this setting, and
we omit the details.

\subsection{Sobolev-Type Embedding and Cauchy Representation}

The next estimate is a decisive consequence of the boundary model
established in Theorem~\ref{t1}. In particular, it shows that once one
passes to the Sobolev level $n\ge1$, the space $H_n^p(\C^+)$ acquires an
$H^\infty$-type control that is absent in the classical Hardy case. This
fact will be crucial in the proof of the generalized Banach algebra
property below.

\begin{prop}
	Let $1\le p\le\infty$ and $n\ge1$. Then
	\[
		H_n^p(\C^+) \hookrightarrow H^\infty(\C^+).
	\]
	More precisely, there exists a constant $0<C\le e^{1/e}$ such that
	\begin{equation}
		\label{e3}
		\|F\|_{H^\infty}\le C\|F\|_{H_n^p},
		\qquad F\in H_n^p(\C^+).
	\end{equation}
\end{prop}
\begin{proof}
	By Theorem~\ref{t1}, it suffices to prove the corresponding estimate for
	boundary values in $H_n^p(\R)\subset W_n^p(\R)$. Since $n\ge1$, the
	classical one-dimensional Sobolev embedding theorem yields
	\[
		\|f\|_{L^\infty(\R)}\le C\|f\|_{W_n^p(\R)}
	\]
	for all $f\in W_n^p(\R)$, where $0<C\le e^{1/e}$; see the footnote in
	\cite[p.~213]{11}. Applying this to the boundary value
	$f=F_l\in H_n^p(\R)$ and using the isometric identification in
	Theorem~\ref{t1}, we obtain
	\[
		\|F\|_{H^\infty}
		=
		\|F_l\|_{L^\infty}
		\le C\|F_l\|_{W_n^p}
		=
		C\|F\|_{H_n^p}.
	\]
	This proves \eqref{e3}.
\end{proof}

We next turn to a Cauchy integral representation of functions in
$H_n^p(\C^\pm)$. This representation is the natural analytic counterpart of
the boundary characterization in Theorem~\ref{t1}, and it serves as the main
tool in the proof of the direct-sum decomposition in the next subsection.

\begin{prop}[Integral representation]
	\label{t6}
	Let $1\le p<\infty$ and $n\in\mathbb N$.
	\begin{enumerate}[label=\textup{(\alph*)}]
		\item A function $F$ belongs to $H_n^p(\C^\pm)$ if and only if there
		exists a unique $f\in \prescript{\pm}{}{H_n^p}(\R)$ such that
		\[
			F_\pm(z)=\frac{\pm1}{2\pi i}\int_{-\infty}^{\infty}
			\frac{f(x)}{x-z}\,dx, \qquad z\in\C^\pm,
		\]
		and then $\lim_{y\to0^\pm}F_\pm(x+iy)=f(x)$ almost everywhere.

		\item If $f\in W_n^p(\R)$ and $p\neq1$, then the functions $F_\pm$
		defined above belong to $H_n^p(\C^\pm)$ and satisfy
		\begin{equation}
			\label{e9}
			\lim_{y\to0^\pm}F_\pm(x+iy)
			=\frac12 f(x)\pm\frac{i}{2}(Hf)(x)
			\qquad \text{a.e.}
		\end{equation}
		For $p=1$, the limit relation \eqref{e9} still holds, but $F_\pm$ need
		not belong to $H_n^1(\C^\pm)$.
	\end{enumerate}
\end{prop}
\begin{proof}
	Part (a) is an immediate consequence of the converse direction of
	Theorem~\ref{t1} applied on $\C^\pm$.

	For part (b), let $f\in W_n^p(\R)$ and define
	\[
		F_\pm(z)=\frac{\pm1}{2\pi i}\int_{-\infty}^{\infty}
		\frac{f(x)}{x-z}\,dx,
		\qquad z\in\C^\pm.
	\]
	These functions are holomorphic in the corresponding half-planes. For
	$k=0$, this is exactly the Cauchy integral itself. For $1\le k\le n$,
	differentiating under the integral sign and using integration by parts, or
	equivalently the weak derivative structure of $f$, we obtain
	\[
		F_\pm^{(k)}(z)
		=
		\pm\frac{k!}{2\pi i}\int_{-\infty}^{\infty}\frac{f(x)}{(x-z)^{k+1}}\,dx
		=
		\pm\frac{(-1)^{k-1}}{2\pi i}
		\int_{-\infty}^{\infty}\frac{D^{(k)}f(x)}{x-z}\,dx.
	\]
	Since $D^{(k)}f\in L^p(\R)$ for $k=0,1,\dots,n$, the classical Hardy-space
	theory implies that $F_\pm^{(k)}\in H^p(\C^\pm)$ whenever $p>1$.
	Therefore $F_\pm\in H_n^p(\C^\pm)$ for $p>1$.

	The boundary relation \eqref{e9} follows from the Plemelj theorem. When
	$p=1$, the same boundary formula remains valid, but $F_\pm$ need not
	belong to $H_n^1(\C^\pm)$. Indeed, if this were true for every
	$f\in W_n^1(\R)$, then the Hilbert transform would map every
	$W_n^1(\R)$-function into $L^1(\R)$, hence in particular every
	compactly supported smooth function into $L^1(\R)$, which is false.
\end{proof}

\subsection{Direct-Sum Decomposition of $W_n^p(\R)$}

For the classical Hardy spaces, every function in $L^p(\R)$ with
$1<p<\infty$ admits a decomposition into the sum of an upper half-plane and
a lower half-plane boundary Hardy function. The next result lifts this
classical decomposition to the Sobolev level. In the Hilbert case it
recovers the known orthogonal decomposition, while for general
$1<p<\infty$ it yields the corresponding Banach-space decomposition
throughout the full reflexive range.

From Proposition~\ref{t6}(b) we obtain
\[
	W_n^p(\R)=\prescript{+}{}{H_n^p}(\R)+\prescript{-}{}{H_n^p}(\R).
\]
Since $\prescript{+}{}{H_n^p}(\R)$ and $\prescript{-}{}{H_n^p}(\R)$ are
subspaces of $\prescript{+}{}{H^p}(\R)$ and $\prescript{-}{}{H^p}(\R)$,
respectively, and since
\[
	\prescript{+}{}{H^p}(\R)\cap\prescript{-}{}{H^p}(\R)=\{0\}
\]
for $1<p<\infty$, the above sum is direct.

\begin{thm}
	\label{t3}
	For $1<p<\infty$ and $n\in\mathbb N$,
	\[
		W_n^p(\R)=\prescript{+}{}{H_n^p}(\R)\oplus\prescript{-}{}{H_n^p}(\R).
	\]
	Moreover, when $p=2$, this decomposition is orthogonal with respect to the
	$W_n^2$ inner product.
\end{thm}
\begin{proof}
	Only the orthogonality for $p=2$ requires comment. If
	$F\in\prescript{+}{}{H_n^2}(\R)$ and
	$G\in\prescript{-}{}{H_n^2}(\R)$, then for each $k=0,1,\dots,n$ we have
	$D^{(k)}F\in\prescript{+}{}{H^2}(\R)$ and
	$D^{(k)}G\in\prescript{-}{}{H^2}(\R)$. Since
	$\prescript{+}{}{H^2}(\R)$ and $\prescript{-}{}{H^2}(\R)$ are mutually
	orthogonal in $L^2(\R)$, it follows that
	\[
		\langle F,G\rangle_{W_n^2}
		=\sum_{k=0}^n \langle D^{(k)}F,D^{(k)}G\rangle_{L^2}=0.
	\]
\end{proof}

\begin{re}
	The decomposition in Theorem~\ref{t3} does not extend to the endpoint cases
	$p=1$ and $p=\infty$. For $p=\infty$, nonzero constant functions belong to
	both $\prescript{+}{}{H_n^\infty}(\R)$ and
	$\prescript{-}{}{H_n^\infty}(\R)$, so the intersection is nontrivial. For
	$p=1$, if such a decomposition were valid, then the Cauchy integral of
	every function in $W_n^1(\R)$ would belong to $H_n^1(\C^+)$, contradicting
	Proposition~\ref{t6}(b).
\end{re}

This decomposition has several useful consequences. First, it is closely
related to adaptive Fourier decomposition (AFD); see \cite{19}. Second, it
is noteworthy from the Banach-space point of view, since not every closed
subspace of a Banach space admits a complement. Finally, even when a
subspace is known to be complemented, an explicit description of a
complementary subspace is often highly nontrivial.

\subsection{Generalized Banach Algebra Structure}

We now arrive at one of the main structural consequences of the previous
results. Once $n\ge1$, the Sobolev-type embedding shows that
$H_n^p(\C^+)$ is stable under pointwise multiplication and, in fact,
carries a generalized Banach algebra structure. This is a genuinely new
phenomenon at the Hardy--Sobolev level: it fails in the classical Hardy
space case $n=0$, and it will play a central role in the operator-theoretic
developments of Section~5.

Recall that a Banach space $(X,\|\cdot\|)$ is called a \textbf{Banach algebra}
if it is equipped with a multiplication making it an algebra and satisfying
\[
	\|xy\|\le \|x\|\,\|y\|
\]
for all $x,y\in X$. More generally, if there exists a constant $C>0$ such
that
\[
	\|xy\|\le C\|x\|\,\|y\|
\]
for all $x,y\in X$, then $(X,\|\cdot\|)$ is called a
\textbf{generalized Banach algebra}.

\begin{thm}
	\label{t4}
	Let $n\in\mathbb N^+$ and $1\le p\le\infty$. Then $H_n^p(\C^+)$,
	equipped with pointwise multiplication, is a generalized Banach algebra.
	More precisely,
	\begin{equation}
		\label{e10}
		\|FG\|_{H_n^p}\le
		\begin{cases}
			\displaystyle
			C\left(\frac{2^{p(n+1)}-1}{2^p-1}\right)^{1/p}
			\|F\|_{H_n^p}\|G\|_{H_n^p},
			& 1\le p<\infty,\\[12pt]
			\displaystyle
			(2^{n+1}-1)\|F\|_{H_n^\infty}\|G\|_{H_n^\infty},
			& p=\infty,
		\end{cases}
	\end{equation}
	where $C$ is the constant in the Sobolev embedding inequality
	\eqref{e3}.
\end{thm}
\begin{proof}
	Assume first that $1\le p<\infty$. For $F,G\in H_n^p(\C^+)$ and
	$0\le k\le n$, the Leibniz rule gives
	\[
		(FG)^{(k)}=\sum_{j=0}^k \binom{k}{j}F^{(k-j)}G^{(j)}.
	\]
	Hence
	\[
		\|(FG)^{(k)}\|_{H^p}
		\le \sum_{j=0}^k \binom{k}{j}\|F^{(k-j)}G^{(j)}\|_{H^p}.
	\]
	Using the embedding inequality \eqref{e3}, we obtain
	\[
		\|F^{(k-j)}G^{(j)}\|_{H^p}
		\le \|F^{(k-j)}\|_{H^\infty}\|G^{(j)}\|_{H^p}
		\le C\|F\|_{H_n^p}\|G\|_{H_n^p},
	\]
	and similarly when the roles of $F$ and $G$ are interchanged. Therefore,
	\[
		\|(FG)^{(k)}\|_{H^p}
		\le C\,2^k\,\|F\|_{H_n^p}\|G\|_{H_n^p}.
	\]
	Raising to the $p$th power and summing over $k=0,\dots,n$ yields
	\eqref{e10} for $1\le p<\infty$.

	For $p=\infty$, Leibniz' rule gives
	\[
		\|(FG)^{(k)}\|_{H^\infty}
		\le \sum_{j=0}^k \binom{k}{j}
		\|F^{(k-j)}\|_{H^\infty}\|G^{(j)}\|_{H^\infty}
		\le 2^k \|F\|_{H_n^\infty}\|G\|_{H_n^\infty}.
	\]
	Summing over $k$ gives
	\[
		\|FG\|_{H_n^\infty}
		\le (2^{n+1}-1)\|F\|_{H_n^\infty}\|G\|_{H_n^\infty},
	\]
	which proves \eqref{e10}.
\end{proof}

\begin{re}
	Theorem~\ref{t4} fails completely in the classical case $n=0$ when
	$1\le p<\infty$. Indeed, if $H^p(\C^+)$ were a generalized Banach algebra
	under pointwise multiplication, then $F^2\in H^p(\C^+)$ would hold for
	every $F\in H^p(\C^+)$, which would imply
	$H^p(\C^+)\subseteq H^{2p}(\C^+)$. This is false. For example,
	\[
		F(z)=\left(\frac{1}{\sqrt{z}\,(z+i)}\right)^{1/p}, \qquad z\in\C^+,
	\]
	belongs to $H^p(\C^+)$ but not to $H^{2p}(\C^+)$.
\end{re}
\section{The Hilbert Case $p=2$}

The boundary-based theory developed in Section~3 applies throughout the full
Banach range $1\le p\le\infty$. In the special case $p=2$, however,
$H_n^2(\C^+)$ carries additional Hilbert-space structure, and this allows a
finer Fourier-analytic description than is available in the general Banach
setting. The purpose of this section is to complement the boundary model of
Section~3 by a frequency-side Hilbert model. The two main results are a
Paley--Wiener theorem identifying $H_n^2(\C^+)$ with a weighted
Fourier-side space on $\R^+$, and an explicit formula for the reproducing
kernel of $H_n^2(\C^+)$.

\subsection{A Paley--Wiener Theorem for $H_n^2(\C^+)$}

We begin with the first main refinement specific to the Hilbert case. While
Section~3 identifies $H_n^2(\C^+)$ through boundary values, in the Hilbert
setting one can go further and realize this space isometrically on the
Fourier side. The natural model is the weighted space $L_n^2(\R^+)$, and
the next theorem shows that the holomorphic Fourier transform provides the
precise Paley--Wiener correspondence between these two spaces.

\begin{thm}[Paley--Wiener]
	\label{t2}
	For every $n\in\mathbb N$, the holomorphic Fourier transform
	\[
		\mathscr{F}(f)(z)=\int_0^\infty f(x)e^{izx}\,dx, \qquad z\in\C^+,
	\]
	defines an isometric isomorphism from $L_n^2(\R^+)$ onto
	$H_n^2(\C^+)$.
\end{thm}
\begin{proof}
	Let $F\in H_n^2(\C^+)$. Since $F\in H^2(\C^+)$, the classical
	Paley--Wiener theorem yields a unique function
	$f\in L^2(\R^+,2\pi\,dx)$ such that
	\[
		F(z)=\int_0^\infty f(x)e^{izx}\,dx, \qquad z\in\C^+.
	\]
	Differentiating under the integral sign, we obtain
	\[
		F^{(k)}(z)=\int_0^\infty (ix)^k f(x)e^{izx}\,dx,
		\qquad k=0,1,\dots,n.
	\]
	Since $F^{(k)}\in H^2(\C^+)$ for each $k$, the classical
	Paley--Wiener theorem again implies that
	$x^k f\in L^2(\R^+,2\pi\,dx)$ for $k=0,1,\dots,n$. Hence
	$f\in L_n^2(\R^+)$.

	Conversely, let $f\in L_n^2(\R^+)$ and define
	\[
		F(z)=\mathscr{F}(f)(z)=\int_0^\infty f(x)e^{izx}\,dx, \qquad z\in\C^+.
	\]
	Then
	\[
		F^{(k)}(z)=\int_0^\infty (ix)^k f(x)e^{izx}\,dx,
		\qquad k=0,1,\dots,n.
	\]
	Since $x^k f\in L^2(\R^+,2\pi\,dx)$, the classical Paley--Wiener theorem
	shows that $F^{(k)}\in H^2(\C^+)$ for each $k$, so
	$F\in H_n^2(\C^+)$.

	Finally, for each $k=0,1,\dots,n$, the classical Paley--Wiener identity
	gives
	\[
		\|F^{(k)}\|_{H^2}^2
		=\int_0^\infty |f(x)|^2\,2\pi x^{2k}\,dx.
	\]
	Summing over $k$ yields
	\[
		\|\mathscr{F}(f)\|_{H_n^2}^2=\|f\|_{L_n^2}^2,
	\]
	which proves that $\mathscr{F}$ is an isometric isomorphism.
\end{proof}

The Paley--Wiener theorem gives a convenient frequency-side model for
$H_n^2(\C^+)$. In particular, combined with Proposition~\ref{p3}, it yields
the following simple characterization.

\begin{cor}
	\label{c:pw-char}
	For every $n\in\mathbb N$,
	\[
		H_n^2(\C^+)
		=
		\bigl\{\,F\in H(\C^+): F\in H^2(\C^+)\ \text{and}\ F^{(n)}\in H^2(\C^+)\,\bigr\}.
	\]
\end{cor}
\begin{proof}
	By Theorem~\ref{t2}, a function $F$ belongs to $H_n^2(\C^+)$ if and only
	if $F=\mathscr{F}(f)$ for some $f\in L_n^2(\R^+)$. By
	Proposition~\ref{p3}, this is equivalent to requiring that both
	$f\in L^2(\R^+,2\pi\,dx)$ and $x^n f\in L^2(\R^+,2\pi\,dx)$. By the
	classical Paley--Wiener theorem, these conditions are equivalent to
	$F\in H^2(\C^+)$ and $F^{(n)}\in H^2(\C^+)$, respectively.
\end{proof}

\begin{re}
	Corollary~\ref{c:pw-char} shows that, in the Hilbert case, the full
	Hardy--Sobolev norm can be recovered from the two extreme conditions
	$F\in H^2(\C^+)$ and $F^{(n)}\in H^2(\C^+)$. This is a particularly simple
	consequence of the Fourier-side model provided by Theorem~\ref{t2}.
\end{re}

\subsection{Reproducing Kernel of $H_n^2(\C^+)$}

The Paley--Wiener theorem above provides an explicit Hilbert-space model for
$H_n^2(\C^+)$. Our second main result in this section is that this model
also leads to an explicit reproducing kernel formula. This gives a concrete
realization of the Hilbertian structure of $H_n^2(\C^+)$ that has no
analogue in the general Banach setting.

Recall that a Hilbert space of functions on a set $E$ is called a
reproducing kernel Hilbert space if point evaluations are bounded. In that
case, for each $z\in E$, there exists a unique kernel function $K_z$ such
that
\[
	F(z)=\langle F,K_z\rangle
\]
for all $F$ in the space. Since $H_n^2(\C^+)\hookrightarrow H^\infty(\C^+)$
by the Sobolev-type embedding established in Section~3, point evaluations
are bounded on $H_n^2(\C^+)$, and hence $H_n^2(\C^+)$ is a reproducing
kernel Hilbert space.

\begin{thm}
	\label{t5}
	The reproducing kernel of $H_n^2(\C^+)$ is given by
	\[
		K_n(z,w)=\frac{1}{2\pi}\int_0^\infty
		\frac{1-x^2}{1-x^{2n+2}}\,e^{ix(w-\overline z)}\,dx,
		\qquad z,w\in\C^+.
	\]
\end{thm}
\begin{proof}
	Fix $z\in\C^+$ and define
	\[
		K_{n,z}(w)
		=
		\frac{1}{2\pi}\int_0^\infty
		\frac{1-x^2}{1-x^{2n+2}}\,e^{ix(w-\overline z)}\,dx.
	\]
	We may write
	\[
		K_{n,z}(w)=\mathscr{F}(g)(w),
	\]
	where
	\[
		g(x)=\frac{1}{2\pi}\,\frac{(1-x^2)e^{-ix\overline z}}{1-x^{2n+2}}.
	\]

	The function
	\[
		x\mapsto \frac{1-x^2}{1-x^{2n+2}}
		=
		\frac{1}{1+x^2+\cdots+x^{2n}}
	\]
	extends continuously to $(0,\infty)$ and is bounded there. Since
	$e^{-ix\overline z}$ decays exponentially like $e^{-x \, Im(z)}$, there exists
	a constant $M>0$ such that
	\[
		|g(x)|\le M e^{-x\, Im(z)},
		\qquad
		|x^n g(x)|\le M e^{-x\,Im(z)},
		\qquad x>0.
	\]
	Hence $g\in L_n^2(\R^+)$ by Proposition~\ref{p3}, and therefore
	$K_{n,z}\in H_n^2(\C^+)$ by Theorem~\ref{t2}.

	Now let $F\in H_n^2(\C^+)$, and let $f\in L_n^2(\R^+)$ be the unique
	function such that
	\[
		F(w)=\int_0^\infty f(x)e^{iwx}\,dx, \qquad w\in\C^+.
	\]
	Then
	\[
		\begin{aligned}
			\langle F,K_{n,z}\rangle_{H_n^2}
			&=\langle \mathscr F(f),\mathscr F(g)\rangle_{H_n^2}
			=\langle f,g\rangle_{L_n^2} \\
			&=\sum_{k=0}^n \int_0^\infty
			f(x)\overline{g(x)}\,2\pi x^{2k}\,dx \\
			&=\int_0^\infty f(x)\,
			\frac{1-x^2}{1-x^{2n+2}}\,e^{ixz}
			\left(\sum_{k=0}^n x^{2k}\right)\,dx.
		\end{aligned}
	\]
	Since
	\[
		\sum_{k=0}^n x^{2k}=\frac{1-x^{2n+2}}{1-x^2},
	\]
	it follows that
	\[
		\langle F,K_{n,z}\rangle_{H_n^2}
		=\int_0^\infty f(x)e^{ixz}\,dx
		=F(z).
	\]
	Thus $K_n$ is the reproducing kernel of $H_n^2(\C^+)$.
\end{proof}

As useful consequences of the explicit kernel formula, we record a uniform
bound for the kernel norms and a corresponding product estimate.

\begin{cor}
	\label{c4}
	Let $n\in\mathbb N^+$ and $z\in\C^+$. Then:
	\begin{enumerate}[label=\textup{(\alph*)}, nosep, wide=0pt, align=left]
		\item
		\[
			\sup_{z\in\C^+}\|K_{n,z}\|_{H_n^2}^2\le \frac14;
		\]
		\item if $F\in H_n^2(\C^+)$ and $G\in H^2(\C^+)$, then
		\[
			\|FG\|_{H^2}
			\le \frac12\,\|F\|_{H_n^2}\,\|G\|_{H^2}
			\le \frac12\,\|F\|_{H_n^2}\,\|G\|_{H_n^2}.
		\]
	\end{enumerate}
\end{cor}
\begin{proof}
	For part (a), by the reproducing property we have
	\[
		\|K_{n,z}\|_{H_n^2}^2=K_n(z,z)
		=
		\frac{1}{2\pi}\int_0^\infty
		\frac{1-x^2}{1-x^{2n+2}}\,e^{-2x \, Im(z)}\,dx.
	\]
	Hence
	\[
		\|K_{n,z}\|_{H_n^2}^2
		\le
		\frac{1}{2\pi}\int_0^\infty
		\frac{1-x^2}{1-x^{2n+2}}\,dx.
	\]
	Since
	\[
		\frac{1-x^2}{1-x^{2n+2}} \le \frac{1}{1+x^2}
		\qquad (x\ge0),
	\]
	it follows that
	\[
		\|K_{n,z}\|_{H_n^2}^2
		\le
		\frac{1}{2\pi}\int_0^\infty \frac{dx}{1+x^2}
		=\frac14.
	\]

	For part (b), we estimate
	\[
		\begin{aligned}
			\|FG\|_{H^2}^2
			&=
			\sup_{y>0}\int_{-\infty}^\infty |F(x+iy)G(x+iy)|^2\,dx \\
			&\le
			\sup_{z\in\C^+}|F(z)|^2\,\|G\|_{H^2}^2 \\
			&=
			\sup_{z\in\C^+}|\langle F,K_{n,z}\rangle_{H_n^2}|^2\,\|G\|_{H^2}^2 \\
			&\le
			\sup_{z\in\C^+}\|K_{n,z}\|_{H_n^2}^2\,\|F\|_{H_n^2}^2\,\|G\|_{H^2}^2 \\
			&\le
			\frac14\,\|F\|_{H_n^2}^2\,\|G\|_{H^2}^2.
		\end{aligned}
	\]
	Taking square roots gives the desired inequality.
\end{proof}

\begin{re}
	The second inequality in Corollary~\ref{c4}(b) is immediate from the
	inclusion $H_n^2(\C^+)\subset H^2(\C^+)$. More generally, if
	$G^{(n)}\in H^2(\C^+)$, the same argument gives
	\[
		\|F\,G^{(n)}\|_{H^2}
		\le \frac12\,\|F\|_{H_n^2}\,\|G^{(n)}\|_{H^2}
		\le \frac12\,\|F\|_{H_n^2}\,\|G\|_{H_n^2},
	\]
	which will be useful in the estimate below.
\end{re}

The kernel estimate also yields a sharper multiplicative bound in the
Hilbert case.

\begin{cor}
	\label{c:sharp-hilbert}
	For every $n\in\mathbb N^+$ and all $F,G\in H_n^2(\C^+)$,
	\[
		\|FG\|_{H_n^2}^2
		\le \frac13\,(4^n-1)\,\|F\|_{H_n^2}^2\,\|G\|_{H_n^2}^2.
	\]
\end{cor}
\begin{proof}
	The case $n=1$ is exactly the result of Kucik~\cite{a1}. Assume the
	statement holds for some $n=k\ge1$. Then for $n=k+1$, the Leibniz rule
	gives
	\[
		\begin{aligned}
			\|FG\|_{H_{k+1}^2}^2
			&=
			\|FG\|_{H_k^2}^2+\|(FG)^{(k+1)}\|_{H^2}^2 \\
			&\le
			\frac13(4^k-1)\|F\|_{H_k^2}^2\|G\|_{H_k^2}^2
			+
			\left[
			\sum_{m=0}^{k+1}\binom{k+1}{m}
			\|F^{(m)}G^{(k+1-m)}\|_{H^2}
			\right]^2.
		\end{aligned}
	\]
	For each $0\le m\le k+1$, we apply Corollary~\ref{c4}(b) to
	$F^{(m)}\in H_{k+1-m}^2(\C^+)$ and $G^{(k+1-m)}\in H^2(\C^+)$. Since
	\[
		\|F^{(m)}\|_{H_{k+1-m}^2}\le \|F\|_{H_{k+1}^2},
		\qquad
		\|G^{(k+1-m)}\|_{H^2}\le \|G\|_{H_{k+1}^2},
	\]
	it follows that
	\[
		\|F^{(m)}G^{(k+1-m)}\|_{H^2}
		\le \frac12\,\|F\|_{H_{k+1}^2}\,\|G\|_{H_{k+1}^2}
	\]
	for all $0\le m\le k+1$. Hence
	\[
		\begin{aligned}
			\|FG\|_{H_{k+1}^2}^2
			&\le
			\frac13(4^k-1)\|F\|_{H_{k+1}^2}^2\|G\|_{H_{k+1}^2}^2
			+
			\left[
			\frac12\sum_{m=0}^{k+1}\binom{k+1}{m}
			\right]^2
			\|F\|_{H_{k+1}^2}^2\|G\|_{H_{k+1}^2}^2 \\
			&=
			\frac13(4^k-1)\|F\|_{H_{k+1}^2}^2\|G\|_{H_{k+1}^2}^2
			+
			4^k\|F\|_{H_{k+1}^2}^2\|G\|_{H_{k+1}^2}^2 \\
			&=
			\frac13(4^{k+1}-1)\|F\|_{H_{k+1}^2}^2\|G\|_{H_{k+1}^2}^2.
		\end{aligned}
	\]
	This completes the induction.
\end{proof}

\subsection{Comparison with $\mathscr{H}_n^2(\C^+)$}

For comparison, we briefly relate $H_n^2(\C^+)$ to the Hilbert-type space
$\mathscr{H}_n^2(\C^+)$ mentioned in the Introduction. Although both spaces
admit Fourier-side descriptions, the corresponding models are substantially
different, and the following examples show that neither space contains the
other.

We now compare $H_n^2(\C^+)$ with the Hilbert-type space
$\mathscr{H}_n^2(\C^+)$ introduced in the Introduction. The Paley--Wiener
description above shows that $H_n^2(\C^+)$ is modeled by the weighted space
$L_n^2(\R^+)$, whereas the corresponding Fourier-side model for
$\mathscr{H}_n^2(\C^+)$ is the space $\mathscr{T}_n^2$ from \cite{a3}.

\begin{cthm}[\cite{a3}]
	Let $n\in\mathbb N^+$. The holomorphic Fourier transform is an isometric
	isomorphism from $\mathscr{T}_n^2$ onto $\mathscr{H}_n^2(\C^+)$, where
	$\mathscr{T}_n^2$ consists of all functions $f:\R^+\to\C$ such that
	$f^{(k)}$ exists for $k=0,1,\dots,n-1$, the derivative $f^{(n-1)}$ is
	absolutely continuous, and each function
	$x\mapsto x^k f^{(k)}(x)$ belongs to $L^2(\R^+)$ for
	$k=0,1,\dots,n$.
\end{cthm}

The first example is the function
\[
	f(x)=
	\begin{cases}
		W(x), & 0<x<1,\\
		0, & x\ge1,
	\end{cases}
\]
where $W$ denotes the everywhere continuous but nowhere differentiable
Weierstrass function~\cite{a23}. Since $x^k f(x)$ is bounded on $(0,1)$ and
vanishes on $[1,\infty)$ for every $k=0,1,\dots,n$, it follows that
$f\in L_n^2(\R^+)$. However, $f$ is nowhere differentiable on $(0,1)$, so
$f\notin\mathscr{T}_n^2$.

Conversely, consider
\[
	g(x)=
	\begin{cases}
		\displaystyle \sum_{m=0}^n (1-x)^m, & 0<x<1,\\[4pt]
		\displaystyle \frac{1}{x}, & x\ge1.
	\end{cases}
\]
Since $g(x)=x^{-1}$ for $x\ge1$, we have $g\notin L^1(\R^+)$, and hence
$g\notin L_n^2(\R^+)$ by Proposition~\ref{c3}. On the other hand,
$g^{(n-1)}$ is absolutely continuous on $(0,\infty)$, and for each
$k=0,1,\dots,n$ the function $x^k g^{(k)}(x)$ is square-integrable on
$\R^+$. Indeed, on $(0,1)$ this is immediate since $g$ is a polynomial,
while on $[1,\infty)$ one has $g^{(k)}(x)=(-1)^k k! \, x^{-k-1}$, so
\[
	x^k g^{(k)}(x)=(-1)^k k!\,x^{-1}\in L^2(1,\infty).
\]
Thus $g\in\mathscr{T}_n^2$.

Hence the spaces $\mathscr{H}_n^2(\C^+)$ and $H_n^2(\C^+)$ intersect, but
neither contains the other.

Finally, we note one further distinction between these two spaces. While
$H_n^2(\C^+)$ is a generalized Banach algebra by Theorem~\ref{t4},
$\mathscr{H}_n^2(\C^+)$ is not.

\begin{prop}
	\label{p:not-gba}
	The space $\mathscr{H}_n^2(\C^+)$ cannot be made into a generalized
	Banach algebra under pointwise multiplication by any equivalent norm.
\end{prop}
\begin{proof}
	Assume, for contradiction, that $\mathscr{H}_n^2(\C^+)$ were a
	generalized Banach algebra. By \cite[Theorems~4.1 and~4.2]{a3}, it is a
	reproducing kernel Hilbert space whose reproducing kernel $K_n$ satisfies
	the two-sided estimate
	\[
		\frac{1}{(n-1)!\sqrt{2n-1}}\frac{1}{\sqrt{|z|}}
		\le
		\|K_{n,z}\|_{\mathscr{H}_n^2}
		\le
		\frac{\sqrt{\pi}}{(n-1)!\sqrt{n}}\frac{1}{\sqrt{|z|}},
		\qquad z\in\C^+.
	\]
	On the other hand, Kucik proved in \cite{a1} that if a Hilbert function
	space over a domain is a generalized Banach algebra under pointwise
	multiplication, then its reproducing kernels must satisfy
	\[
		\sup_{z\in\Omega}\|K_z\|_{\mathscr H}<\infty.
	\]
	But the lower bound above shows that
	\[
		\|K_{n,z}\|_{\mathscr{H}_n^2}\to\infty
		\qquad\text{as } |z|\to0^+,
	\]
	which is a contradiction. Therefore $\mathscr{H}_n^2(\C^+)$ is not a
	generalized Banach algebra.
\end{proof}

The Hilbertian description developed in this section complements the
boundary-based function theory established in Section~3. We now turn to
operator-theoretic applications on $H_n^p(\C^+)$.

\section{Operator Theory on $H_n^p(\C^+)$}

In this section we turn to operator-theoretic consequences of the function
theory developed in Section~3. In particular, the point-evaluation
structure and the generalized Banach algebra property obtained there provide
the basic tools for studying multiplication operators and weighted
composition operators on $H_n^p(\C^+)$. Our main result in the first part of
the section is a precise spectral description of multiplication operators
induced by multipliers. In the second part, we establish two useful
sufficient conditions for the boundedness of weighted composition operators.

\subsection{Multiplication Operators}

Let
\[
	\mathscr{M}_{n,p}
	=
	\left\{
		\psi\in H(\C^+) : \psi F\in H_n^p(\C^+)\ \text{for all } F\in H_n^p(\C^+)
	\right\}
\]
denote the multiplier space of $H_n^p(\C^+)$. For each
$\psi\in\mathscr{M}_{n,p}$, the associated multiplication operator is
defined by
\[
	T_\psi F=\psi F, \qquad F\in H_n^p(\C^+).
\]

We begin with the basic boundedness of multiplication operators and several
elementary structural properties of the multiplier space.

\begin{prop}
	Let $\psi\in\mathscr{M}_{n,p}$. Then the multiplication operator
	$T_\psi:H_n^p(\C^+)\to H_n^p(\C^+)$ is bounded.
\end{prop}
\begin{proof}
	Since $H_n^p(\C^+)$ is a Banach space, it suffices by the closed graph
	theorem to show that the graph of $T_\psi$ is closed. Let $\{F_m\}$ be a
	sequence in $H_n^p(\C^+)$ such that
	\[
		F_m\to0 \quad\text{and}\quad T_\psi(F_m)\to G
	\]
	in $H_n^p(\C^+)$ for some $G\in H_n^p(\C^+)$, and prove that $G=0$.

	For $n\ge1$, boundedness of point evaluations follows from the Sobolev-type
	embedding established in Section~3; for $n=0$, it is part of the classical
	Hardy-space theory. Thus, for each fixed $z\in\C^+$, there exists a constant
	$C_z>0$ such that
	\[
		|F(z)|\le C_z\|F\|_{H_n^p}
		\qquad\text{for all } F\in H_n^p(\C^+).
	\]
	Hence
	\[
		\begin{aligned}
			|G(z)|
			&\le |G(z)-\psi(z)F_m(z)|+|\psi(z)F_m(z)| \\
			&\le C_z\|G-T_\psi(F_m)\|_{H_n^p}
			+|\psi(z)|\,C_z\|F_m\|_{H_n^p}\to0
		\end{aligned}
	\]
	as $m\to\infty$. Since $z\in\C^+$ was arbitrary, we conclude that
	$G\equiv0$. Therefore the graph of $T_\psi$ is closed, and hence $T_\psi$
	is bounded.
\end{proof}

The space $\mathscr{M}_{n,p}$ is naturally equipped with the operator norm
\[
	\|\psi\|_{\mathscr M}:=\|T_\psi\|_{\mathscr B(H_n^p)}.
\]
Thus the boundedness problem for multiplication operators is equivalent to
describing the multiplier space itself.

Recall that for each $z\in\C^+$, the evaluation functional
\[
	M_z:H_n^p(\C^+)\to\C,\qquad F\mapsto F(z),
\]
is bounded. The next proposition shows that these point evaluations play the
role of eigenvectors for the adjoint multiplication operator.

\begin{prop}
	\label{p5}
	If $\psi\in\mathscr{M}_{n,p}$, then
	\[
		T_\psi^*(M_z)=\psi(z)M_z
		\qquad\text{for every } z\in\C^+.
	\]
\end{prop}
\begin{proof}
	For any $F\in H_n^p(\C^+)$,
	\[
		(T_\psi^*M_z)(F)=M_z(T_\psi F)=M_z(\psi F)=\psi(z)F(z)=\psi(z)M_z(F),
	\]
	which proves the claim.
\end{proof}

\begin{re}
	In a reproducing kernel Hilbert space, bounded multiplication operators are
	often detected through the fact that kernel functions are eigenvectors of
	the adjoint. Proposition~\ref{p5} is the corresponding point-evaluation
	version of this observation in the present setting.
\end{re}

We next record several basic inclusion relations for the multiplier space.

\begin{prop}
	\label{p6}
	Let $1\le p\le\infty$ and $n\in\mathbb N$.
	\begin{enumerate}[label=\textup{(\alph*)}, nosep, wide=0pt, align=left]
		\item $H_n^\infty(\C^+)\hookrightarrow \mathscr{M}_{n,p}$;
		\item $H_n^p(\C^+)\hookrightarrow \mathscr{M}_{n,p}$ whenever $n\ge1$;
		\item $\mathscr{M}_{n,p}\hookrightarrow H^\infty(\C^+)$.
	\end{enumerate}
\end{prop}
\begin{proof}
	For (a), if $n=0$, then for every $F\in H^p(\C^+)$,
	\[
		\|\psi F\|_{H^p}\le \|\psi\|_{H^\infty}\|F\|_{H^p},
	\]
	so $\psi\in\mathscr M_{0,p}$. If $n\ge1$, the same conclusion follows from
	Leibniz' rule together with the estimate used in the proof of
	Theorem~\ref{t4}, replacing one factor by $\psi$ and controlling all
	derivatives of $\psi$ by the $H_n^\infty$ norm.

	Part (b) is an immediate consequence of Theorem~\ref{t4}, since
	$H_n^p(\C^+)$ is a generalized Banach algebra for $n\ge1$.

	For (c), Proposition~\ref{p5} implies that each value $\psi(z)$ belongs to
	the point spectrum of $T_\psi^*$, hence to the spectrum of $T_\psi$. Thus
	\[
		|\psi(z)|\le r(T_\psi)\le \|T_\psi\|=\|\psi\|_{\mathscr M},
		\qquad z\in\C^+,
	\]
	and therefore $\psi\in H^\infty(\C^+)$.
\end{proof}

\begin{re}
	In general, none of the inclusions in Proposition~\ref{p6} is reversible.
	For instance, when $p=2$ and $n=1$:
	\begin{enumerate}[label=\textup{(\alph*)}, nosep]
		\item $H_1^2(\C^+)\not\subset H_1^\infty(\C^+)$;
		\item $\mathscr M_{1,2}\not\subset H_1^2(\C^+)$, since nonzero constants
		belong to $\mathscr M_{1,2}$ but not to $H_1^2(\C^+)$;
		\item $H^\infty(\C^+)\not\subset \mathscr M_{1,2}$.
	\end{enumerate}
	A necessary and sufficient condition for membership in $\mathscr M_{n,2}$ is
	given in \cite[Theorem~2]{a1}.
\end{re}

To obtain the spectral description of multiplication operators, we need the
following invertibility lemma.

\begin{lem}
	\label{l1}
	If $\psi\in\mathscr{M}_{n,p}$ and
	\[
		\inf_{z\in\C^+}|\psi(z)|>0,
	\]
	then $1/\psi\in\mathscr{M}_{n,p}$.
\end{lem}
\begin{proof}
	Clearly $\psi\in H^\infty(\C^+)$ by Proposition~\ref{p6}(c), and
	$1/\psi\in H^\infty(\C^+)$ because $\inf_{z\in\C^+}|\psi(z)|>0$.

	Let $F\in H_n^p(\C^+)$. By the higher-order quotient rule
	\cite[Corollary~5]{a7},
	\[
		D^k\!\left(\frac{F}{\psi}\right)
		=
		\frac{(-1)^k}{\psi^{k+1}}
		\sum_{j=0}^k (-1)^j\binom{k+1}{j}\,
		\psi^j D^k\!\left(\psi^{\,k-j}F\right),
		\qquad k=0,1,\dots,n.
	\]
	Since $\psi\in\mathscr M_{n,p}$, the multiplication operator $T_\psi$ is
	bounded on $H_n^p(\C^+)$. Hence, by iteration, multiplication by
	$\psi^m$ is bounded on $H_n^p(\C^+)$ for every integer $m\ge1$. It follows
	that $\psi^{k-j}F\in H_n^p(\C^+)$, and therefore
	\[
		D^k(\psi^{k-j}F)\in H^p(\C^+).
	\]
	Since both $\psi$ and $1/\psi$ are bounded, each term on the right-hand side
	belongs to $H^p(\C^+)$, and hence $D^k(F/\psi)\in H^p(\C^+)$ for all
	$k=0,1,\dots,n$. Therefore $(1/\psi)F\in H_n^p(\C^+)$, proving that
	$1/\psi\in\mathscr M_{n,p}$.
\end{proof}

The next theorem is the main result of this subsection. It gives a complete
spectral description of multiplication operators on $H_n^p(\C^+)$ and shows
that the spectrum is determined exactly by the range of the multiplier
symbol.

\begin{thm}
	\label{t7}
	Let $\psi\in\mathscr{M}_{n,p}$. Then
	\[
		\sigma(T_\psi)=\overline{\psi(\C^+)}.
	\]
\end{thm}
\begin{proof}
	By Proposition~\ref{p5}, every value $\psi(z)$ is an eigenvalue of
	$T_\psi^*$, hence belongs to $\sigma(T_\psi)$. Therefore
	\[
		\overline{\psi(\C^+)}\subseteq \sigma(T_\psi).
	\]

	Conversely, let $\lambda\notin \overline{\psi(\C^+)}$. Then there exists
	$\delta>0$ such that
	\[
		|\lambda-\psi(z)|\ge \delta
		\qquad\text{for all } z\in\C^+.
	\]
	Applying Lemma~\ref{l1} to $\lambda-\psi$, we obtain
	\[
		\frac1{\lambda-\psi}\in\mathscr{M}_{n,p}.
	\]
	It is then immediate that
	\[
		(\lambda I-T_\psi)^{-1}=T_{1/(\lambda-\psi)}.
	\]
	Hence $\lambda\notin \sigma(T_\psi)$, so
	\[
		\sigma(T_\psi)\subseteq \overline{\psi(\C^+)}.
	\]
	This proves the theorem.
\end{proof}

An immediate and useful consequence of Theorem~\ref{t7} is that compact
multiplication operators are completely rigid.

\begin{cor}
	Let $\psi\in\mathscr{M}_{n,p}$. Then $T_\psi$ is compact if and only if
	$\psi\equiv0$.
\end{cor}
\begin{proof}
	If $T_\psi$ is compact, then its spectrum is countable and has no
	accumulation point except possibly $0$. By Theorem~\ref{t7},
	\[
		\sigma(T_\psi)=\overline{\psi(\C^+)},
	\]
	so $\overline{\psi(\C^+)}$ must be countable. Since a nonconstant holomorphic
	function maps open sets onto open sets, the open mapping theorem forces
	$\psi$ to be constant. Since $H_n^p(\C^+)$ is infinite-dimensional and
	$0\in\sigma(T_\psi)$ for every compact operator on an infinite-dimensional
	Banach space, that constant must be $0$.
\end{proof}

\subsection{Weighted Composition Operators}

We now turn to weighted composition operators
\[
	T_{\psi,\varphi}F=\psi\cdot(F\circ\varphi),
\]
where $\psi\in H(\C^+)$ and $\varphi\in\mathscr S(\C^+)$ is an analytic
self-map of $\C^+$. In contrast to multiplication operators, a complete
description of bounded weighted composition operators on $H_n^p(\C^+)$ seems
to be substantially more delicate. We therefore restrict ourselves to two
practical sufficient conditions for boundedness, both of which are readily
verifiable in concrete situations.

Our first criterion is based on a lower bound involving the derivative of
the symbol $\varphi$.

\begin{thm}
	\label{t8}
	Let $n\in\mathbb N^+$. Assume that
	\[
		\psi\in H_n^\infty(\C^+) \quad \text{or} \quad \psi\in H_n^p(\C^+),
	\]
	and that $\varphi\in\mathscr S(\C^+)$ satisfies
	$\varphi'\in H_{n-1}^\infty(\C^+)$ and the function
	\[
		A_\varphi(z,w)=\bigl|Re(\varphi'(z))+i\,Im(\varphi'(w))\bigr|,
		\qquad z,w\in\C^+,
	\]
	has a positive lower bound. Then $T_{\psi,\varphi}$ is bounded on
	$H_n^p(\C^+)$.
\end{thm}
\begin{proof}
	Using the Fa\`a di Bruno formula for holomorphic composition, we may write
	\[
		(F\circ\varphi)^{(k)}
		=
		\sum_{a_1+2a_2+\cdots+ka_k=k}
		c_{(a_1,\dots,a_k)}
		(F^{(m_k)}\circ\varphi)
		(\varphi')^{a_1}\cdots(\varphi^{(k)})^{a_k},
	\]
	where $m_k=a_1+\cdots+a_k$. Applying Leibniz' rule to
	$\psi(F\circ\varphi)$, we obtain
	\[
		\begin{aligned}
			\|(\psi(F\circ\varphi))^{(k)}\|_{H^p}
			\le
			\sum_{j=0}^k
			\sum_{a_1,\dots,a_j}
			\binom{k}{j}c_{(a_1,\dots,a_j)}
			\|F^{(m_j)}\circ\varphi\|_{H^p}
			\|\varphi'\|_{H^\infty}^{a_1}\cdots
			\|\varphi^{(j)}\|_{H^\infty}^{a_j}
			\|\psi^{(k-j)}\|_{H^\infty}.
		\end{aligned}
	\]
	Since $\varphi'\in H_{n-1}^\infty(\C^+)$, all derivatives
	$\varphi^{(j)}$ with $1\le j\le n$ belong to $H^\infty(\C^+)$. Moreover, the
	hypothesis on $A_\varphi$ implies the boundedness of the composition
	operator $C_\varphi:F\mapsto F\circ\varphi$ on $H^p(\C^+)$ for
	$1\le p<\infty$; see \cite[Theorem~2.8]{a9}. For $p=\infty$ this is
	immediate. Finally, $\psi^{(k-j)}\in H^\infty(\C^+)$ follows either from
	$\psi\in H_n^\infty(\C^+)$ or from the Sobolev embedding when
	$\psi\in H_n^p(\C^+)$.

	Therefore each derivative of $\psi(F\circ\varphi)$ up to order $n$ belongs
	to $H^p(\C^+)$, so $T_{\psi,\varphi}$ maps $H_n^p(\C^+)$ into itself. Since
	$H_n^p(\C^+)$ is Banach, the closed graph theorem implies that
	$T_{\psi,\varphi}$ is bounded.
\end{proof}

We also use the standard Julia--Carath\'eodory notion of angular derivative
at infinity \cite{a8}. For $\varphi\in\mathscr S(\C^+)$, we say that
$\varphi$ has a finite angular derivative at infinity if
\[
	\sup_{z\in\C^+}\frac{Im(z)}{Im(\varphi(z))}<\infty.
\]
For $0<p<\infty$, this is equivalent to the boundedness of the composition
operator $C_\varphi:F\mapsto F\circ\varphi$ on $H^p(\C^+)$.

This yields a second useful sufficient condition.

\begin{thm}
	\label{t9}
	Let $n\in\mathbb N^+$. Assume that
	\[
		\psi\in H_n^\infty(\C^+) \quad \text{or} \quad \psi\in H_n^p(\C^+),
	\]
	and that $\varphi\in\mathscr S(\C^+)$ satisfies
	$\varphi'\in H_{n-1}^\infty(\C^+)$ and has a finite angular derivative at
	infinity. Then $T_{\psi,\varphi}$ is bounded on $H_n^p(\C^+)$.
\end{thm}
\begin{proof}
	The proof is identical to that of Theorem~\ref{t8}. The only difference is
	that the boundedness of $C_\varphi$ on $H^p(\C^+)$ for $1\le p<\infty$ now
	follows from the Elliott--Jury theorem \cite{a8}, while the case
	$p=\infty$ remains immediate.
\end{proof}

\begin{re}
	For $1\le p<\infty$, a basic example of a symbol satisfying either of the
	above sufficient conditions is the affine map
	\[
		\varphi(z)=kz+b,
		\qquad k>0, \ Im(b)>0.
	\]
	It is also worth emphasizing that the assumption
	$\psi\in H_n^p(\C^+)$, and not only $\psi\in H_n^\infty(\C^+)$, is
	admissible because of the Sobolev embedding theorem established in
	Section~3.
\end{re}

We conclude with a brief remark on the Hilbert case $p=2$. In that setting,
the boundedness of weighted composition operators on $H_n^2(\C^+)$ also
admits a characterization in terms of positivity of an associated kernel;
see \cite[Theorem~7.1]{4}. We record this criterion for completeness, but do
not use it in what follows.

A kernel \( A:\mathbb{C}^+\times\mathbb{C}^+\to\mathbb{C} \) is called
\textbf{non-negative} if
\[
\sum_{i,j=1}^{N} c_i\overline{c_j}\,A(z_i,z_j)\ge 0
\]
for every finite choice of points \( z_1,\dots,z_N\in\mathbb{C}^+ \) and
scalars \( c_1,\dots,c_N\in\mathbb{C} \). Thus the operator
$T_{\psi,\varphi}$ is bounded on $H_n^2(\C^+)$ if and only if there exists
a constant $M\ge0$ such that the kernel
\[
	A_n(z,w)
	=
	\frac{1}{2\pi}\int_0^\infty
	\Bigl[
	M^2e^{ix(w-\overline z)}
	-
	\overline{\psi(w)}\psi(z)\,
	e^{ix(\varphi(w)-\overline{\varphi(z)})}
	\Bigr]
	\frac{1-x^2}{1-x^{2n+2}}\,dx
\]
is non-negative on $\C^+\times\C^+$.
	\section*{Acknowledgements}
	This work was supported by the NSF of Guangdong Province, China (Grant No. 2025A1515011213) and the Science and Technology Development Fund of Macau SAR (No. 0020/2023/RIB1) . 
	
	  \bibliographystyle{elsarticle-harv}
	  \bibliography{reference}
	
	
	
	
\end{document}